\documentclass[a4paper,12pt]{amsart}

\usepackage{amsmath}
\usepackage{amsthm}
\usepackage{amssymb}
\usepackage{graphicx}
\usepackage{cases}
\usepackage{bm}
\usepackage{euscript}
\newcommand{\R}{\mathbb{R}}
\newcommand{\D}{\mathbb{D}}
\newcommand{\C}{\mathbb{C}}

\renewcommand{\H}{\mathcal{H}}
\renewcommand{\S}{\mathcal{S}}
\newcommand{\N}{\mathbb{N}}
\numberwithin{equation}{section}
\newcommand{\abs}[1]{\lvert#1\rvert}
\def\th{\theta}
         
      \def \ov{\overline}         \def\beq{\begin{equation}} \def\beqq{\begin{equation*}}
\def\eeq{\end{equation}}    \def\eeqq{\end{equation*}} \def\bprof{\begin{proof}}    \def\eprof{\end{proof}}  \def\bald{\begin{aligned}}    \def\eald{\end{aligned}}
         \def\bthm{\begin{thm}}
\def\ethm{\end{thm}}   \def\p{\partial}    
\theoremstyle{cupthm}
\newtheorem{thm}{Theorem}[section]

\newtheorem{cor}[thm]{Corollary}
\newtheorem{lem}[thm]{Lemma}
\theoremstyle{cupdefn}

\theoremstyle{cuprem}
\newtheorem{rmk}[thm]{Remark}
\numberwithin{equation}{section}
\newtheorem{innercustomthm}{Theorem}
\newenvironment{customthm}[1]
  {\renewcommand\theinnercustomthm{#1}\innercustomthm}
  {\endinnercustomthm}

\begin{document}
\title{Note on the second derivative of bounded analytic functions}
\author{Gangqiang Chen}
\address{School of Mathematics and Computer Sciences,
Nanchang University, Nanchang 330031, China}
\address{Graduate School of Information Sciences, Tohoku University, Sendai 980-8579, Japan}
\email{cgqmath@qq.com; chenmath@ncu.edu.cn}
\date{\today}



\begin{abstract}
Assume  $z_0$ lies in the open unit disk $\mathbb{D}$ and  $g$ is an analytic self-map of $\mathbb{D}$. We will determine the region of values of $g''(z_0)$ in terms of $z_0$, $g(z_0)$ and the hyperbolic derivative of $g$ at $z_0$, and give the form of all the extremal functions. In particular,  we obtain a smaller sharp upper bound for $|g''(z_0)|$ than Ruscheweyh's inequality for the case of the second derivative. Moreover, we use a different method to obtain Sz{\'a}sz's inequality, which provides a sharp upper bound for $|g''(z_0)|$ depending only on $|z_0|$.
\end{abstract}

\subjclass[2010]{primary 30C80; secondary 30F45}
\keywords{bounded analytic functions, Schwarz-Pick lemma, Dieudonn\'e's lemma, Peschl's invariant derivative, hyperbolic derivative. }

\maketitle
\section{Introduction}
\label{intro}
We denote by $\C$ the complex plane and define the disks $\D(c, \rho)$ and $\overline{\D}(c, \rho)$ by
$\D(c, \rho):=\left\{ \zeta \in \C : |\zeta-c|< \rho \right\}$,
and
$\overline{\D}(c, \rho):=\left\{\zeta \in \C : |\zeta-c|\le \rho \right\}$ for $c\in\C$ and $\rho>0$. The open and closed unit disks $\D(0,1)$ and $\overline{\D}(0,1)$ are denoted by $\mathbb{D}$ and $\overline{\D}$ respectively. Throughout this article,  $\H$ denotes the class of all analytic self-mappings of $\mathbb{D}$, and $\mathcal{S}$  denotes  the Schur class, i.e. the set of analytic functions from $\D$ to $\overline{\D}$.

First we recall a classical result obtained by Schwarz in 1890, which
says that if $g \in \mathcal{H}$ satisfies $g(0)=0$, then $| g(z_0)| \le |z_0|$ for any non-zero $z_0$ in $\D$ and $|g'(0)|\le 1$, and equalities hold if and only if
 $ g(z)=e^{i \theta}z$ for some $\theta \in \mathbb{R}$.
Since then, more and more authors started to consider the space $\H$ and obtained many relevant extensions of Schwarz's Lemma. Rogosinski\cite{rogosinski1934} determined the variability region of $g(z_0)$ for $z_0\in\D$, $g\in \H$ with $g(0)=0$ and $|g'(0)|<1$, which says that the region of values of $g(z_0)$ is the closed disk $\overline{\D}(c,\rho)$, where
 $$c=\dfrac{z_0g'(0)(1-|z_0|^2)}{1-|z_0|^2|g'(0)|^2},\quad \rho= \dfrac{|z_0|^2(1-|g'(0)|^2)}{1-|z_0|^2|g'(0)|^2}.$$
Rogosinski's Lemma
can be considered as a
sharpened version of Schwarz's Lemma (see also \cite{duren1983univalent} and \cite{goluzin1969geometric}).

 In 1916, the well-known Schwarz-Pick Lemma was proved in \cite{Pick1916}, which states that
 $$|g'(z)|\le \frac{1-|g(z)|^2}{1-|z|^2},\quad g\in \mathcal{H}, \quad z\in \D,$$
 and equality holds if and only if $g$ is a conformal automorphism of $\D$,  i.e.,
$g(z)=e^{i\theta }(z-a)/(1-\bar a z)$, $\theta\in \R$, $a \in \D$.
If we define $T_{a}\in \text{Aut}(\D)$ by
$$T_{a}(z)=\frac{z+a}{1+\overline{a}z},\quad a,z\in \D,$$
then the
Schwarz-Pick Lemma can be shown as follows.
\begin{customthm}{A (The Schwarz-Pick Lemma)}\label{thm:g-first}
Let $z_0, \delta_0\in \D$.
Suppose that $g\in \H$,
$g(z_0)=\delta_0$.
Set
$$
 g_{\alpha}(z) = T_{\delta_0}\big(\alpha T_{-z_0}(z)\big).$$
Then the region of values of $g'(z_0)$ is the closed disk
$$
\overline{\D}(0, \frac{1-|\delta_0|^2}{1-|z_0|^2})
=\{g'_{\alpha}(z_0):\alpha\in\overline{\D} \},
$$
and $g(z)$ is the form of
$T_{\delta_0}\big(T_{-z_0}(z) g^*(z)\big)$,
where $g^*\in\S$.
Further, $g'(z_0)\in \partial\D(0, \dfrac{1-|\delta_0|^2}{1-|z_0|^2})$ if and only if
$g(z)=g_{\alpha}(z)$ for some constant $\alpha \in \partial \D$.
\end{customthm}
A consequence of the Schwarz-Pick Lemma is the sharp inequality  $|g'(z_0)|\le 1/(1-|z_0|^2)$ for $g\in\H$, $z_0\in \D$, and equality occurs only for $g(z)=e^{i \theta}(z-z_0)/(1-\overline{z}_0 z)$, $\th \in \R$. Sz\'asz \cite{szasz1920} extended this inequality and obtained a sharp upper bound for $|g''|$ (see also \cite{avkhadiev2009schwarz}).
In addition, Ruscheweyh \cite{ruscheweyh1985two} proved that, for $g\in \mathcal{H}$ and $n \in \mathbb{N}$, the approximately sharp inequality
\begin{equation}\label{f^n}
 |g^{(n)}(z)|\le \frac{n!(1-|g(z)|^2)}{(1-|z|)^n(1+|z|)},\quad z\in \D,
 \end{equation}
is valid (see also \cite{Anderson2006,DaiPan2008}).

In 1931, Dieudonn\'e \cite{dieudonne1931} explicitly described the variability region of $f'(z_0)$, for $f\in \mathcal{H}$ with $f(0)=0$, at a fixed point $z_0\in \mathbb{D}$,
which could be considered as Schwarz's Lemma for the derivative $f'$.
We show
Dieudonn\'e's Lemma as follows(see also \cite{beardon2004multi}, \cite{chen_2019}, \cite{kaptanoglu2002refine}).
\begin{customthm}{B (Dieudonn\'e's Lemma)}
Let $z_0,w_0\in \D$ with $|w_0|<|z_0|$. Suppose that $f\in\mathcal{H}$ satisfies $f(0)=0$ and $f(z_0) = w_0$.
Set $u_0=w_0/z_0$, $f_{\alpha}(z)=zT_{u_0}(\alpha T_{-z_0}(z))$. Then the region of values of $f'(z_0)$ is the closed disk
\begin{align*}
\overline{\mathbb{D}}
  \left( c_1(z_0,w_0) , \rho_1(z_0,w_0) \right)=\{f'_{\alpha}(z_0):\alpha\in \overline{\D}\},
\end{align*}
where $$c_1=c_1(z_0,w_0)=\frac{w_0}{z_0}, \qquad \rho_1=\rho_1(z_0,w_0)=\frac{|z_0|^2-|w_0|^2}{|z_0|(1-|w_0|^2)},$$
and $f(z)$ is the form of
$z T_{u_0}\left(T_{-z_0}(z) f^*(z)\right)$,
where $f^*\in \S$.
Further,
$f'(z_0) \in \partial \D \left( c_1, \rho_1 \right) $ if and only if
$f(z)=f_{\alpha}(z)$ for some constant $\alpha\in \partial \D$.
\end{customthm}
The important point  is that Dieudonn\'e's Lemma is an application of the Schwarz-Pick Lemma. To see this, let $g(z)=f(z)/z$, then $g$ is an analytic self-mapping of $\D$.  Apply the Schwarz-Pick Lemma to $g$ and note that $g'(z)=(zf'(z)-f(z))/z^2$, we can easily obtain Dieudonn\'e's Lemma.

In 1996, Mercer \cite{mercer1997sharpened} obtained a description of the variability region of $g(z)$ for $z\in\D$, $g\in \H$ with $g(0)=0$ and $g(z_0)=w_0 (z_0\ne 0)$, which says that the variability region of $g(z)$ is the closed disk
$\overline{\D}(c',\rho')$, where
$$c'=\frac{z_0 u_0(1-|T_{-z_0}(z)|^2)}{1-|u_0|^2|T_{-z_0}(z)|^2},\qquad \rho'=\frac{|z_0T_{-z_0}(z)| (1-|u_0|^2)}{1-|u_0|^2|T_{-z_0}(z)|^2},\quad \text{and}\quad u_0=\frac{w_0}{z_0}. $$
It is worth pointing out that Dieudonn\'e's Lemma and Rogosinski's Lemma are the limiting cases of Mercer's result as $z\to z_0$ and $z_0\to 0$ respectively. In recent years, more and more articles on regions of variability and Schwarz-Pick type estimates  have been written \cite{Ponnusamy2007univalent,
Ponnusamy2009satisfying,yanagihara2005,yanagihara2010families,Zhong2021}.

It is natural to determine the variability region of $g''(z_0)$ for $g\in\H$, $z_0\in\D$, and further give pointwise sharp estimates for $g''(z_0)$.
We need to recall the notion of higher-order hyperbolic derivatives for $g\in \H$, which was recently introduced by P. Rivard \cite{rivard2011higher-orderHyperbolicDerivatives} (see also \cite{Baribeau2013} and \cite{rivard2013application}).
For $ z, w\in\D,$ let $[z,\,w]$ be defined by
\begin{equation}\label{eq:dist}
[z,\,w] = \frac{z-w}{1-\overline{w}z},
\end{equation}
and let $[z,\,z]=0$ for $z\in\partial\D$
to extend the definition, then
we can define a holomorphic function $\Delta_{z_0}g$ on $\D$ by
\begin{equation}\label{eq:hdquo}
\Delta_{z_0}g(z)=
\begin{cases}
\dfrac{[g(z),\,g(z_0)]}{[z,\,z_0]} &\quad \text{for}~z\ne z_0, \\
\null & \\
\dfrac{(1-|z_0|^2)g'(z_0)}{1-|g(z_0)|^2} &\quad\text{for}~z=z_0.
\end{cases}
\end{equation}
The result of the above construction is an operator $\Delta_{z_0}$, which maps every function $g\in \H$ to the function $\Delta_{z_0} g\in \S$. Thus we will have the chain rule $\Delta_{z_0} (f\circ g)=(\Delta_{g(z_0)} f)\circ g\cdot \Delta_{z_0} g$ (cf.~ \cite{Baribeau-Rivard2009}).
Therefore, we can iterate the process and construct the hyperbolic divided difference of order $j$ of the function $g$ for distinct parameters $z_0,\cdots,z_{j-1}$ as follows:
$$
g_j(z;z_{j-1},\cdots,z_0)=(\Delta_{z_{j-1}}\circ\cdots\circ\Delta_{z_0})g(z).
$$
We can also write the definition of $g_j(z;z_{j-1},\cdots,z_0)$ recursively as
\begin{equation}
g_{j}(z;z_j,\cdots,z_0)=\frac{[g_{j-1}(z;z_{j-1},\cdots,z_0),g_{j-1}(z_j;z_{j-1},\cdots,z_0)]}{[z,z_j]}.
\end{equation}
Using a limiting process, we can define the hyperbolic divided difference of order $n$ of $g$ with parameter $z$ by
\begin{equation}\label{def:Delta-z-f-zeta}
\Delta_z^n g(\zeta)=\Delta^n g(\zeta;z,\cdots,z):=\lim_{z_{n-1}\to z}\cdots\lim_{z_0\to z}\Delta^n g(\zeta;z_{n-1},\cdots,z_0)\quad(\zeta\neq z).
\end{equation}
Then we can define the $n$-th order  hyperbolic derivative  of $g\in \H$ at the point $z\in \D$ by
$$H^n g(z):=\Delta_z^n g(z)=\Delta^n g(z;z,\cdots,z):=\lim_{\zeta\to z}\Delta^n g(\zeta;z,\cdots,z).$$
We note that the usual hyperbolic derivative coincides with the first-order hyperbolic derivative $H^{1} g$, which means
$$g^h(z):=\frac{(1-|z|^2)g'(z)}{1-|g(z)|^2}=H^{1} g(z).$$
In addition, the inequality $|H^1g(z)|\le 1$ is equivalent to Ruscheweyh's inequality \eqref{f^n} for the first derivative $g'$.
Then an interesting question arises: can we also obtain Ruscheweyh's inequality \eqref{f^n} for the second derivative $g''$?

We organize the remainder of this paper as follows.
In Sect. 2, we will  provide an explicit description of the variability region
$\{g''(z_0): g\in \mathcal{H}, g(z_0) =\delta_0, H^1(z_0)=\delta_1\}$ for given $z_0$, $\delta_0$, $\delta_1$, and give the form of all the extremal functions.
In Sect. 3, we will apply the result in Sect. 2 to
obtain a best possible upper bound for $g''(z_0)$, which is a
shaper version of Ruscheweyh's result. We also use a different method to derive
Sz\'asz's result \cite{szasz1920}, which gives a sharp upper bound for $|g''(z_0)|$ depending only on $|z_0|$.
 Our work could motivate further investigations on the variability regions and estimates of higher-order derivatives of  bounded analytic functions.

\section{Variability region of the second derivative}
We begin this section with the introduction to Peschl's invariant derivatives.
For $g:\mathbb{D}\to \mathbb{D}$ holomorphic, the so-called Peschl's invariant derivatives $D^n g(z)$ are defined by the Taylor series expansion \cite{peschl1955invariants} (see also \cite{kim2007invariant} and \cite{Schippers07}):
$$z\mapsto h(z):=\frac{g(\dfrac{z+z_0}{1+\overline{z}_0 z})-g(z_0)}{1-\overline{g(z_0)}g(\dfrac{z+z_0}{1+\overline{z}_0 z})}=\sum_{n=1}^{\infty}\frac{D^n g(z_0)}{n!}z^n,\quad z, z_0\in \D.$$
The relations $D^n g(z_0)=h^{(n)}(0)$ immediately follow from the definition
of $D^ng.$

Precise forms of $D^n g(z)$, $n=1,2$, are expressed by
\begin{align*}
D_1 g(z)&=\frac{(1-\abs{z}^2)g'(z)}{1-\abs{g(z)}^2},\\
D_2 g(z)&=\frac{(1-\abs{z}^2)^2}{1-\abs{g(z)}^2}
\Bigg[g''(z)-\frac{2\overline{z}g'(z)}{1-\abs{z}^2}
+\frac{2\overline{g(z)}g'(z)^2} {1-\abs{g(z)}^2}\Bigg].
\end{align*}

We saw above that $H^{1} g(z)=D^{1} g(z)$, and
\[H^{2} g(z)=\frac{D^{2} g(z)}{2\left(1-\left|H^{1} g(z)\right|^{2}\right)},\]
for $g\in \H$ not a Blaschke product of degree $\leq 1$.
We can rewrite the inequality $|H^{2}g(z)|\leq 1$ in terms of Peschl's
invariant derivatives, which implicitly shows an inequality for $g''(z)$.
\begin{lem}[Yamashita $\text{\cite[Theorem 2]{Yamashita1994}}$]
If $g:\mathbb{D}\to \mathbb{D}$ is holomorphic, then
\begin{equation}\label{eq:yamainequality}
|D_2g(z)|\le 2(1-|D_1g(z)|^2),\quad z\in\D,
\end{equation}
equality holds for a point $z\in\D$ if and only if
$g$ is a Blaschke product of degree at most $2.$
\end{lem}
We remark that a function $B(z)$ is called a Blaschke product of degree $n \in \N$  if it takes the form
         $$ B(z)=e^{i \theta}\prod\limits_{j=1}^{n}
         \frac{z-z_j}{1-\overline{z_j}z}, \quad z, z_j\in \D, \theta \in \mathbb{R}.$$

We can now state our first main result on the variability region of the second derivative of bounded analytic functions.
Before that we denote
$c_2$ and $\rho_2$ by
\begin{equation}
\left\{
\begin{aligned}
c_2=c_2(z_0,\delta_0,\delta_1)&= \frac{2(1-|\delta_0|^2)}{(1-|z_0|^2)^2}
(\overline{z_0}-\overline{\delta_0}\delta_{1})\delta_{1};\\
\rho_2=\rho_2(z_0,\delta_0,\delta_1)&=\frac{2(1-|\delta_0|^2)}{(1-|z_0|^2)^2}(1-|\delta_1|^2).
\end{aligned}
\right.
\end{equation}

\begin{thm}\label{thm:g-second}
Let $z_0, \delta_0\in \D$, and $ \delta_1 \in \ov{\D}$.
Suppose that $g\in \H$,
$g(z_0)=\delta_0$ and  $H^1 g(z_0)=\delta_1$.
Set
\begin{align*}
 &g_{\delta_1}(z) = T_{\delta_0}\left(\delta_1 T_{-z_0}(z)\right),\\
& g_{\delta_1,\alpha}(z)=T_{\delta_0}\left(T_{-z_0}(z) T_{\delta_1}(\alpha T_{-z_0}(z))\right).
\end{align*}
\begin{enumerate}
\item If $|\delta_1|=1$, then $g''(z_0)=c_2$ and
    $g(z)=g_{\delta_1} (z)$.
\item If $|\delta_1|<1$, then the region of values of $g''(z_0)$ is the closed disk
$$
\overline{\D}(c_2, \rho_2)
=\{g''_{\delta_1,\alpha}(z_0):\alpha\in\overline{\D} \},
$$
and $g(z)$ is the form of
$T_{\delta_0}\left(T_{-z_0}(z) T_{\delta_1}(T_{-z_0}(z)g^*(z))\right)$,
where $g^*\in\S$.
Further, $g''(z_0)\in \partial\D(c_2, \rho_2)$ if and only if
$g(z)=g_{\delta_1,\alpha}(z)$ for some constant $\alpha \in \partial \D$.
In particular, equality in $|g''(z_0)|\le |c_2|+\rho_2$ holds if and only if $\delta_1=0$ or $g(z)=g_{\delta_1,\alpha}(z)$, where $\alpha=e^{i \theta}$, $\theta=\arg{(\overline{z_0}\delta_{1}-\overline{\delta_0}\delta_{1}^2)}
$.
\end{enumerate}
\end{thm}
\begin{proof}
Note that \eqref{eq:yamainequality} is equivalent to
\begin{equation}\label{eq:g''}
|g''(z_0)-c_2|\le |\rho_2|,
\end{equation}
and equality holds here if and only if $g$ is a Blaschke product of degree 1 or 2 satisfying $g(z_0)=\delta_0$ and $H^1g(z_0)=\delta_1$.
Case (1) follows from the same method as in the proof of \cite[Lemma 2.2]{chen_2019}.
We just need to prove Case (2).

From the relation between $H^1 g$ and $g'$, we obtain
\begin{equation}\label{eq:g-derivatives-condition}
g'(z_0)=\frac{1-|\delta_0|^2}{1-|z_0|^2}\delta_1.
\end{equation}
Inequality \eqref{eq:g''} shows that $g''(z_0)\in\overline{\D}(c_2, \rho_2)$.
Next, we show the closed disk $\overline{\D}(c_2,\rho_2)$ is covered.
Let $g(z)=g_{\delta_1,\alpha}(z)$, $\alpha\in\overline{\D}$, so that $g(z_0)=\delta_0$.
Direct calculations give  \eqref{eq:g-derivatives-condition} and
$g''(z_0)=c_2+\rho_2\alpha$. The closed disk $\overline{\D}(c_2,
\rho_2)$ is covered since $\alpha \in \overline{\D}$ is arbitrary.

Next we determine the form of $g(z)$. We will do so by constructing explicitly the function $g$ as follows. Let $g_0:=g$, define the functions $g_1,g_2$ by
$$g_{j+1} (z) = \frac{[g_{j} (z),g_{j}(z_0)]}
    {[z,z_0]},\quad j=0,1,$$
where $g_2$ is analytic in $\D$ with $|g_2(z)|\le 1$ for $z\in \D$. From the construction we know that $g_1(z_0)=H^{1}g(z_0)=\delta_1$ and $g_j(z)=T_{\delta_j}(T_{-z_0}(z) g_{j+1} (z))$
 for $j=0,1$. Therefore, we have
 \begin{equation}\label{eq:g}
g(z)=T_{\delta_0}\big(T_{-z_0}(z) T_{\delta_1}( T_{-z_0}(z)g_2(z))\big),
 \end{equation}
where $g_2$ belongs to $\S$.

We know that $g''(z_0)\in\p\D(c_2, \rho_2)$ if and only if $g$ is a Blaschke product of degree 2 satisfying $g(z_0)=\delta_0$ and $H^1g(z_0)=\delta_1$. Since $g$ is the form of \eqref{eq:g}, differentiate twice both sides of \eqref{eq:g} and then substitute $z = z_0$, we obtain
$g''(z_0)=c_2+\rho_2 g_2(z_0)$. Thus $g''(z_0)\in\p\D(c_2, \rho_2)$ leads to $g_2(z_0)=\alpha\in\p\D$. It follows that
$g_2(z)=\alpha\in\p\D$ and
$$g(z)=T_{\delta_0}\big(T_{-z_0}(z) T_{\delta_1}(\alpha T_{-z_0}(z))\big), \quad \alpha\in \p\D.$$

We continue to prove the necessary and sufficient condition for equality in $|g''(z_0)|\le |c_2|+\rho_2$ to hold.
For $0<|\delta_1|<1$, we already know that equality in $|g''(z_0)-c_2|\le \rho_2$ holds if and only if $g(z)=g_{\delta_1,\alpha}(z)$, $\alpha\in \partial{\D}$. By basic geometry, equality in $|g''(z_0)|\le |c_2|+\rho_2$ holds if and only if $g''(z_0)\in \partial \D(c_2, \rho_2)$ and $\arg{g''(z_0)}=\arg{c_2}$, or if and only if $g''(z_0)=t c_2$, where $t=1+\rho_2/|c_2|$.

If $g''(z_0)=t c_2$ for $t=1+\rho_2/|c_2|$, then $g''(z_0)=c_2+c_2\rho_2/
|c_2|\in \partial \D(c_2, \rho_2)$, thus $g(z)=g_{\delta_1,\alpha}(z)$, $\alpha\in \partial{\D}$.
We remain to determine the precise value of $\alpha$.
A calculation shows that $g''(z_0)=c_2+\rho_2\alpha$. By comparison, we can conclude that $$\alpha=\frac{c_2}{|c_2|}=\dfrac{(\overline{z_0}-\overline{\delta_0}\delta_{1})\delta_{1}}
{|\overline{z_0}-\overline{\delta_0}\delta_{1}||\delta_{1}|}.$$

Conversely, if $g(z)=g_{\delta_1,\alpha}(z)$, $\alpha=c_2/|c_2|$, then a simple direct calculation yields
$g''(z_0)=tc_2$ for $t=1+\rho_2/|c_2|$.
\end{proof}
Based on Theorem \ref{thm:g-second}, we can obtain the second-order Dieudonn\'e's Lemma as its corollary \cite{chen_2019}(see also \cite{cho2012multi}). Denote $c'_2(z_0,w_0,\delta_1)$ and $\rho'_2(z_0,w_0,\delta_1)$ by
\begin{equation*}
\left\{
\begin{aligned}
c'_2=c'_2(z_0,w_0,\delta_1)&=\frac{2(r^2-s^2)}{r^2(1-r^2)^2}
\delta_1(1-\frac{z_0\overline{w_0}}{\overline{z_0}}\delta_1),\\ \rho'_2=\rho'_2(z_0,w_0,\delta_1)&=\frac{2(r^2-s^2)}{r(1-r^2)^2}(1-|\delta_1|^2).
\end{aligned}
\right.
\end{equation*}

\begin{cor}[The Second-Order Dieudonn\'e's Lemma]
Let $z_0, w_0\in \mathbb{D}$, $\delta_1\in \ov{\D}$ with $|w_0|=s<r=|z_0|$,
$$w_1= c_1(z_0,w_0)+\rho_1(z_0,w_0)\dfrac{r\delta_1}{\overline{z_0}}.$$
 Suppose that $f\in\mathcal{H}$, $f(0)=0$, $f(z_0) = w_0$ and $f'(z_0)=w_1$.
Set $u_0=w_0/z_0$, and
\begin{align*}
&f_{\delta_1}(z)=z T_{u_0}(\delta_1 T_{-z_0}(z));\\
&f_{\delta_1,\alpha}(z)=z T_{u_0}\big(T_{-z_0}(z) T_{\delta_1}(\alpha T_{-z_0}(z))\big).
\end{align*}
Then
\begin{enumerate}
\item If $|\delta_1|=1$, then $f''(z_0)=c'_2$ and $f(z)=f_{\delta_1}(z)$.
\item If $|\delta_1|<1$, then the region of values of $f''(z_0)$ is the closed disk
$$\overline{\D}(c'_2, \rho'_2)
=\{f''_{\delta_1,\alpha}(z_0):\alpha\in \overline{\D}\},
$$
and $f(z)$ is the form of
$$z T_{u_0}\left(T_{-z_0}(z) T_{\delta_1}(f^*(z) T_{-z_0}(z))\right),$$
where $f^*\in\S$.
Further, $f''(z_0)\in \p\D(c'_2,\rho'_2)$ if and only if
$f(z)=f_{\delta_1,\alpha}(z)$ for some constant $\alpha\in \partial \D$.
\end{enumerate}
\end{cor}
We mention that the present author and Yanagihara \cite{chen2020} applied this consequence to precisely determine the variability region
$V(z_0,w_0)=\{f''(z_0):f\in\H,f(0)=0, f(z_0)=w_0\}$.

\section{Estimates of the second derivative}
Theorem \ref{thm:g-second} can also give a sharp upper bound for $|g''(z)|$, $g\in \H$, depending only on $|z|$ and $|g(z)|$.

\begin{thm}\label{thm2:g-second}
  Suppose that $z_0$ and $\delta_0$ are points in $\mathbb{D}$ with $|\delta_0|=R$, $|z_0|=r$. If $g\in\mathcal{H}$ satisfies $g(z_0) = \delta_0$, then

\begin{numcases}{|g''(z_0)|\le}
\dfrac{2(1-R^2)(r+R)}{(1-r^2)^2}, & $r+2R\ge2$;\label{rR2}\\
\dfrac{(1+ R)(4-4R+r^2)}{2(1-r^2)^2},    & $r+2R <2$.  \label{rR1}
\end{numcases}
Equality holds in \eqref{rR2} if and only if
$$g(z) = e^{i\theta }\frac{{z - a}}{1 - \overline a z},$$
where
$$ \theta=\arg(-\bar z_0\delta_0),\quad a=\frac{r+R}{r(1 +rR)}z_0.$$
Equality holds in \eqref{rR1} if and only if

$$g(z)=e^{i \theta} \frac{u^2+\frac{1}{2}z_0 u -\frac{R}{r^2}z_0^2}{1+\frac{1}{2}\overline{z}_0 u -\frac{R}{r^2}\overline{z}_0^2 u^2 },\quad u=\frac{z-z_0}{1-\overline{z}_0 z},\quad \theta=\arg(-\overline{z_0}^2 \delta_0),\quad z\in \D.$$
(If $z_0=0$, then $g(z)=T_{\delta_0}\big(e^{i \theta} z^2\big)$, $\theta \in \R $ \text{is arbitrary}.)
\end{thm}
\begin{proof}[Proof.]

From  Theorem \ref{thm:g-second}, we have
\begin{equation}\label{eq:g-second}
   \begin{aligned}
    | g''(z_0)|&\le |c_2|+\rho_2=\frac{2(1-R^2)}{(1-r^2)^2}\big(|\delta_1||\bar z_0-\overline{\delta_0}\delta_1|+1-|\delta_1|^2)\big)\\
        &\le\frac{2(1-R^2)}{(1-r^2)^2}\big(|\delta_1|(r+R|\delta_1|)+1-|\delta_1|^2\big)\\
       &=\frac{2(1-R^2)}{(1-r^2)^2}\Psi(x),
    \end{aligned}
\end{equation}
where
$$\Psi(x)=(R-1)x^2+rx+1,\quad x=|\delta_1|\in (0,1],$$
and equality holds in the second last inequality if and only if $rR\delta_1=-\overline{z_0}\delta_0|\delta_1|$.

  If $\delta_0=0$, then $\Psi(x)$ takes its maximum at $x=r/2<1$. In this case,
\begin{equation}\label{eq:R=0}
|g''(z_0)|\le \frac{2\Psi(r/2)}{(1-r^2)^2}
=\dfrac{4+r^2}{2(1-r^2)^2},
\end{equation}
and the equality holds if and only if $g(z)=T_{-z_0}(z) T_{\delta_1}(\alpha T_{-z_0}(z))$, where $|\delta_1|=r/2$ and $\alpha \in \partial \D$ is arbitrary.

In the following proof,  we suppose that $\delta_0\neq 0$.
Observe that $\Psi(x)$ takes its maximum at $x=r/(2(1-R))$, which is less than $1$ if and only if $r+2R<2$. In this case, the sharp upper bound for $|g''(z_0)|$ is
$$\frac{2(1-R^2)}{(1-r^2)^2}\Psi(\frac{r}{2(1-R)})
=\frac{2(1-R^2)}{(1-r^2)^2}\big(1+\frac{r^2}{4(1-R)}\big)
=\dfrac{(1+ R)(4-4R+r^2)}{2(1-r^2)^2}.$$

If $z_0\neq 0$, then the sharp upper bound for $|g''(z_0)|$ is obtained if and only if $g(z)=T_{\delta_0}\left(T_{-z_0}(z) T_{\delta_1}(\alpha T_{-z_0}(z))\right)$, where $ \alpha=-\dfrac{\bar z_0 \delta_0}{z_0 R}$, and $\delta_1=-\dfrac{\bar z_0\delta_0}{2R(1-R)}$. In other words, equality holds in \eqref{rR1} if and only if the form of $g$ is
$$g(z)=e^{i \theta} \frac{u^2+\frac{1}{2}z_0 u -\frac{R}{r^2}z_0^2}{1+\frac{1}{2}\overline{z}_0 u -\frac{R}{r^2}\overline{z}_0^2 u^2 },\quad u=T_{-z_0}(z),\quad z\in \D,\quad \theta= \arg(-\overline{z_0}^2 \delta_0).$$

We remark in this case that
$$
\frac{2(1-R^2)}{(1-r^2)^2}\big(1+\frac{r^2}{4(1-R)}\big)< \frac{2(1-R^2)}{(1-r^2)^2}\big(1+\frac{r}{2}\big)
\le \frac{2(1-R^2)}{(1-r^2)^2}\big(1+r\big).
$$

If $z_0=0$, then
\begin{equation}\label{eq:g''0}
|g''(0)|\le 2(1-R^2),
\end{equation}
and
equality holds if and only if $g(z)=T_{\delta_0}\left(T_{-z_0}(z) T_{\delta_1}(\alpha T_{-z_0}(z))\right)$, where $\delta_1=0$ and $\alpha \in \partial \D$, or if and only if
 $$g(z)=T_{\delta_0}\big(\alpha z^2\big),\quad \alpha \in \partial \D.$$

For $r+2R \ge 2$, $\Psi(x)\le \Psi(1)=r+R$ in the interval $0\le x \le 1$, so that
$$ |g''(z_0)|\le \frac{2(1-R^2)\Psi(1)}{(1-r^2)^2}
  = \frac{2(1-R^2)}{(1-r^2)^2}(r+R).$$
Equality holds in the above inequality if and only if $g(z)=T_{\delta_0}(\delta_1 T_{-z_0}(z))$, where $\delta_1=-\bar{z}_0\delta_0/(rR)$. That is to say, equality holds in \eqref{rR2} if and only if $g$ is a Blaschke product of degree 1 of the following form
$$g(z) = e^{i\theta }\frac{{z - a}}{1 - \overline a z},$$
where
$$\theta=\arg(-\overline{z_0} \delta_0),\quad a=\frac{r+R}{r(1 +rR)}z_0.$$
We remark in this case that
$$\frac{2(1-R^2)}{(1-r^2)^2}(r+R)<\frac{2(1-R^2)}{(1-r^2)^2}(1+r).$$
\end{proof}
\begin{rmk}
Rivard asked a question in \cite{rivard2013application} that could we obtain Ruscheweyh's inequality \eqref{f^n} by using $|H^n g(z)|\le 1$ for $n\ge 2$ and $g\in\H$. However,
in the case of order 2, the inequality \eqref{f^n} shows that
$$|g''(z_0)|\le \frac{2(1-|g(z_0)|^2)}{(1-|z_0|)^2(1+|z_0|)},\quad g\in \H.$$
It is worth noting that Theorem \ref{thm2:g-second} offers a smaller upper bound for $|g''(z_0)|$, which essentially indicates that $|H^2 g(z)|\le 1$ cannot yield \eqref{f^n} for the case of second derivative.
\end{rmk}

From Theorem \ref{thm2:g-second}, we can obtain a sharp upper bound for the modulus of the second derivative of $g\in \H$ depending only on $|z|$, which offers a different method to Sz\'asz's original proof \cite{szasz1920}.
 \begin{cor}[\cite{szasz1920}]\label{cor:g-second}
 If $g \in \mathcal{H}$, then
 \begin{equation}\label{eq:g''z}
 |g''(z)|\le \frac{(8+|z|^2)^2 }{32(1 - |z|^2)^2}.
 \end{equation}
Equality holds in \eqref{eq:g''z} at some $z_0\in \D$ if and only if
$$g(z)=e^{i \theta} \frac{8u^2+4z_0 u -z_0^2}{8+4\overline{z}_0 u -\overline{z}_0^2 u^2 },\quad u=\frac{z-z_0}{1-\overline{z}_0 z},\quad \theta \in \R.$$
\end{cor}
\begin{proof}
If $g(z)=0$ and $H^1 g(z)=1$, then $g$ is a Blaschke product of degree 1 and
$$|g''(z)|\le \frac{2|z|}{(1-|z|^2)^2}.$$
If $g(z)=0$ and $H^1 g(z)<1$, then from \eqref{eq:R=0} we can obtain
$$|g''(z)|\le
\frac{4+|z|^2}{2(1-|z|^2)^2}.$$

Next we consider $g(z)\neq 0$.
For $z_0\in \D$, let $\delta_0=g(z_0)$, $R=|\delta_0|$ and $r=|z_0|$.
If $r=0$, then equality in \eqref{eq:g''z} holds if and only if $$g(z)=e^{i\theta }z^2,\quad \theta \in \R.$$
Suppose that $r\ne 0$.
From Theorem \ref{thm2:g-second}, we consider two cases for $r+2R\ge 2$ and $r+2R<2$, respectively.

\emph{Case (i).}  For $r+2R\ge 2$, we know that
  $$|g''(z_0)|\le \dfrac{2(1-R^2)(r+R)}{(1-r^2)^2}=\frac{2\varphi(R)}{( 1 - r^2 )^2},$$
where $\varphi(R)=-R^3-rR^2+r+R$ and $1-r/2\le R<1$. Let
$$\varphi'(R)=-3R^2-2rR+1=0.$$
Then we obtain two roots of the equation above,
$$R_1=-\frac{1}{\sqrt{3+r^2}-r},\quad
R_2=\frac{1}{\sqrt{3+r^2}+r}.$$
Note that $R_1<0$, while $R_2< 1-r/2$ is equivalent to $r^2+4r-8< 0$, which is always valid for $r\in (0,1)$. Thus, $\varphi(R)$ is decreasing in the interval $1-r/2\le R<1$. In this case,
\begin{equation}\label{eq:r+2R>2}
|g''(z_0)|\le \frac{2\varphi(1-\dfrac{r}{2})}{( 1 - r^2 )^2} =\frac{r(2+r)(4-r)}{4( 1 - r^2 )^2}.
\end{equation}
Equality holds if and only  $g(z)=T_{\delta_0}(\alpha T_{-z_0}(z))$, where $\alpha=-\bar{z}_0\delta_0/(rR)$, $R=1-r/2$. In other words, equality holds in \eqref{eq:r+2R>2}  if and only if
$$g(z) = e^{i\theta }\frac{z - a}{1 - \overline a z},$$ where
$$ a=\frac{2+r}{r(2+2r-r^2)}z_0,\quad\theta \in \R.$$

\emph{Case (ii).}  For $r+2R< 2$, we note that
$$|g''(z_0)|\le \frac{(1+ R)(4-4R+r^2)}{2(1-r^2)^2}=\frac{\Phi(R)}{2(1-r^2)^2},$$
where
$$\Phi(R)=-4R^2+r^2R+r^2+4.$$
This is a parabolic function with an axis of symmetry of $R=r^2/8$, which is always  less than $1-r/2$ for $r\in [0,1)$. Therefore, $\Phi(R)$ reaches its maximum at $R=r^2/8$,
\begin{equation}\label{r+2R<2}
|g''(z_0)|\le \frac{\Phi(r^2/8)}{2(1-r^2)^2}=\frac{(8 +r^2)^2}{32(1-r^2)^2}.
\end{equation}
Equality holds if and only if
$g(z)=T_{\delta_0}\big(T_{-z_0}(z) T_{\delta_1}(e^{i \theta} T_{-z_0}(z))\big)$,
where $ \theta=\arg(-\overline{z_0}^2 \delta_0)$,  $|\delta_0|=r^2/8$,  and $\delta_1=-\bar z_0\delta_0/(2R(1-R))$.
In other words,  equality holds in \eqref{r+2R<2} if and only if the form of $g$ is
$$g(z)=e^{i \theta} \frac{8u^2+4z_0 u -z_0^2}{8+4\overline{z}_0 u -\overline{z}_0^2 u^2 },\quad u=T_{-z_0}(z),\quad z\in \D,\quad \theta \in \R.$$
Noting that for $0\le r<1$,
$$\frac{r(2+r)(4-r)}{4( 1 - r^2 )^2}<\frac{(8 +r^2)^2}{32(1-r^2)^2},$$
and
$$\frac{4+r^2}{2(1-r^2)^2}\le\frac{(8 +r^2)^2}{32(1-r^2)^2},$$
we see that inequality \eqref{eq:g''z} holds.
\end{proof}


\section*{Acknowledgement}
The author would like to express his deep gratitude to Prof. Toshiyuki Sugawa for his valuable comments and instructive suggestions. 
\bibliographystyle{srtnumbered}

\end{document}